\documentclass{amsart}

\newtheorem{theorem}{Theorem}[section]
\newtheorem{lemma}[theorem]{Lemma}
\newtheorem{proposition}[theorem]{Proposition}
\theoremstyle{definition}

\theoremstyle{remark}
\newtheorem{remark}[theorem]{Remark}

\newtheorem*{formulas}{Formulas}

\numberwithin{equation}{section}

\begin{document}

\title[A Study of Submanifolds of the Complex Grassmannian Manifold]{A Study of Submanifolds of the Complex Grassmannian Manifold with Parallel Second Fundamental Form}


\author{Isami Koga}
\address{Faculty of Mathematics, KYUSHU UNIVERSITY 744 Motooka, Nishi-ku, Fukuoka, 819-0395, JAPAN}
\curraddr{}
\email{i-koga@math.kyushu-u.ac.jp}
\thanks{}

\author{Yasuyuki Nagatomo}
\address{Department of Mathematics, MEIJI UNIVERSITY, 
Higashi-Mita, Tama-ku, Kawasaki-shi, Kanagawa 214-8571, JAPAN}
\curraddr{}
\email{yasunaga@meiji.ac.jp}
\thanks{}


\date{}

\dedicatory{}

\commby{}

\begin{abstract}
We prove an extension of a theorem of A.Ros on {\it a characterization of seven compact Kaehler submanifolds by holomorphic pinching} \cite{cite6}  to certain submanifolds of the complex Grassmannian manifolds.
\end{abstract}

\maketitle


\section{Introduction}
Let $\mathbb{C}P^n(1)$ be the $n$-dimensional complex projective space with the constant holomorphic sectional curvature 1 and $M^m$ an $m$-dimensional compact K\"{a}hler submanifold immersed in $\mathbb{C}P^n(1)$. In \cite{cite6} Ros proved that the holomorphic sectional curvature of $M$ is greater than or equal to $\frac{1}{2}$ if and only if $M$ has the parallel second fundamental form. Our goal in the present paper is to extend this resut to submanifolds immersed in the complex Grassmannian manifold.

Let $Gr_p(\mathbb{C}^n)$ be the complex Grassmannian manifold of complex $p$-planes in $\mathbb{C}^n$. Since the tautological bundle $S \rightarrow Gr_p(\mathbb{C}^n)$ is a subbundle of a trivial bundle $Gr_p(\mathbb{C}^n)\times \mathbb{C}^n \rightarrow Gr_p(\mathbb{C}^n)$, we obtain a quotient bundle $Q \rightarrow Gr_p(\mathbb{C}^n)$. This is called the universal quotient bundle. We notice the fact that the holomorphic tangent bundle $T_{1,0}M$ over $Gr_p(\mathbb{C}^n)$ can be identified with the tensor product of holomorphic vector bundles $S^*$ and $Q$, where $S^* \rightarrow Gr_p(\mathbb{C}^n)$ is the dual bundle of $S \rightarrow Gr_p(\mathbb{C}^n)$. If $\mathbb{C}^n$ has a Hermitian inner product, $Gr_p(\mathbb{C}^n)$ has a standard metric of Fubini-Study type and $S$, $Q$ have Hermitian metrics and Hermitian connections. In the present paper, we prove the following theorem:

\begin{theorem}
Let $Gr_p(\mathbb{C}^n)$ be the complex Grassmannian manifold of complex $p$-planes in $\mathbb{C}^n$ with a standard metric $h_{Gr}$ induced from a Hermitian inner product on $\mathbb{C}^n$ and $f$ a holomorphic isometric immersion of a compact K\"{a}hler manifold $(M, h_M)$ with a Heremitian metric $h_M$. We denote by $Q\rightarrow Gr_p(\mathbb{C}^n)$ the universal quotient bundle over $Gr_p(\mathbb{C}^n)$ of rank $q$.  We assume that the pull-back bundle of $Q \rightarrow Gr_p(\mathbb{C}^n)$ is projectively flat. Then the holomorphic sectional curvature of $M$ is greater than or equal to $\frac{1}{q}$ if and only if $M$ has the parallel second fundamental form.
\end{theorem}
\begin{remark}
We can suppose that $p\geq q$ without loss of generality because $Gr_p(\mathbb{C}^n) \cong Gr_{n-p}(\mathbb{C}^{n^*})$. In fact we can show that there is no immersion satisfying projectively flatness in the case that
$p < q$. (See Remark 3.1.)
\end{remark}
\begin{remark}
When $q=1$, Theorem 1.1 is the same as  a theorem of Ros \cite{cite6}. In this case, $f^*Q\rightarrow M$ is projectively flat for any submanifold because $Q\rightarrow Gr_p(\mathbb{C}^n)$ is of rank 1. The sufficient condition in our theorem is that the holomorphic sectional curvature is greater than or equal to 1, which is distinct from $\frac{1}{2}$ in a theorem of Ros. This is because we take a metric of Fubini-Study type with constant holomorphic sectional curvature 2.
\end{remark}


\section{Preliminaries}
Let $Gr_p(\mathbb{C}^n)$ be the complex Grassmannian manifold of complex $p$-planes in $\mathbb{C}^n$ with a standard metric $h_{Gr}$ induced from a Hermitian inner product on $\mathbb{C}^n$. We denote by $S$ the tautological vector bundle over $Gr_p(\mathbb{C}^n)$. Since $S \rightarrow Gr_p(\mathbb{C}^n)$ is a subbundle of a trivial vector bundle $\underline{\mathbb{C}^n}=Gr_p(\mathbb{C}^n) \times \mathbb{C}^n \rightarrow Gr_p(\mathbb{C}^n)$, we obtain a holomorphic vector bundle $Q \rightarrow Gr_p(\mathbb{C}^n)$ as a quotient bundle. This is called the universal quotient bundle over $Gr_p(\mathbb{C}^n)$. For simplicity, it is called the quotient bundle. Consequently we have a short exact sequence of vector bundles:
 \begin{eqnarray}
 0 \rightarrow S \rightarrow \underline{\mathbb{C}^n} \rightarrow Q \rightarrow 0. \nonumber
 \end{eqnarray}
Taking the orthogonal complement of  $S$ in $\underline{\mathbb{C}^n}$ with respect to the Hermitian inner product on $\mathbb{C}^n$, we obtain a complex subbundle $S^{\perp}$ of $\underline{\mathbb{C}^n}$. As  $C^{\infty}$ complex vector bundle, $Q$ is naturally isomorphic to $S^{\perp}$. Consequently, the vector bundle $S \rightarrow Gr_p(\mathbb{C}^n)$ (resp. $Q \rightarrow Gr_p(\mathbb{C}^n)$) is equipped with a Hermitian metric $h_S$ (resp.$h_Q$) and so a Hermitian connection $\nabla^S$ (resp.$\nabla^Q$). The holomorphic tangent bundle over $Gr_p(\mathbb{C}^n)$ is identified with $S^* \otimes Q\rightarrow Gr_p(\mathbb{C}^n)$ and the Hermitian metric on the holomorphic tangent bundle is induced from the tensor product of $h_{S^*}$ and $h_Q$.

Let $w_1$,$\cdots$,$w_n$ be a unitary basis of $\mathbb{C}^n$. We denote by $\mathbb{C}^p$ the subspace of $\mathbb{C}^n$ spanned by $w_1$,$\cdots$,$w_p$ and by $\mathbb{C}^q$ the orthogonal complement of $\mathbb{C}^p$, where $q=n-p$. The orthogonal projection to $\mathbb{C}^p$, $\mathbb{C}^q$ is denoted by $\pi_q$, $\pi_q$ respectively. Let $G$ be the special unitary group $SU(n)$ and $P$ the subgroup $S\left( U(p) \times U(q)\right)$ of $SU(n)$ according to the decomposition. Then, $Gr_p(\mathbb{C}^n) \cong G/P$. The vector bundles $S$,$Q$ are identified with $G\times_P\mathbb{C}^p$,  $G\times_P\mathbb{C}^q$ respectively. We denote by $\Gamma(S)$, $\Gamma(Q)$  spaces of sections of $S$, $Q$ respectively. Let $\pi_Q:\mathbb{C}^n \rightarrow \Gamma(Q)$ be a linear map defined by
\begin{eqnarray}
\pi_Q(w) := [g, \pi_q(g^{-1}w)] \in G\times_P\mathbb{C}^q  ,\qquad w\in \mathbb{C}^n, g\in G. \nonumber
\end{eqnarray}
The bundle injection $i_Q:Q\rightarrow \underline{\mathbb{C}^n}$ can be expressed by the following :
\begin{eqnarray}
i_Q([g,v]) = ([g], gv) ,	\qquad v \in \mathbb{C}^q,\ g\in G,\ [g] \in Gr_p(\mathbb{C}^n) \cong G/P. \nonumber
\end{eqnarray}
Let $t$ be a section of $Q\rightarrow Gr_p(\mathbb{C}^n)$. Since $i_{Q}(t)$ can be regarded as a $\mathbb{C}^n$-valued function $t:Gr_p(\mathbb{C}^n)\rightarrow \mathbb{C}^n$, $\pi_Qd(i_Q(t))$ defines a connection on $Q$. This is nothing but $\nabla ^Q$.

Similarly, we can write a bundle injection $i_S:S\rightarrow \underline{\mathbb{C}^n}$ and a linear map $\pi _S:\mathbb{C}^n\rightarrow \Gamma(S)$:
\begin{align}
i_S([g,u]) &= ([g], gu),& \qquad &u\in \mathbb{C}^p,\ g\in G,\ [g]\in G/P, \nonumber \\
\pi_S(w) :&= [g, \pi_p(g^{-1}w)],&  &w\in \mathbb{C}^n,\ g\in G. \nonumber
\end{align}
The connection $\pi_Sd(i_S(s))$ on $S$ is nothing but $\nabla^S$.

We introduce the second fundamental form $H$ in the sense of Kobayashi \cite{cite2}, which is a (1,0)-form with values in ${\rm Hom}(S,Q) \cong S^*\otimes Q$ :
\begin{eqnarray}
di_S(s) = \nabla^Ss + H(s), \quad H(s) = \pi_Qd(i_S(s)), \qquad s\in \Gamma(S).
\end{eqnarray}
Similarly, we introduce the second fundamental form $K$, which is a (0,1)-form with values in ${\rm Hom}(Q,S)\cong Q^*\otimes S$ :
\begin{eqnarray}
di_Q(t) = K(t) + \nabla^Qt, \quad K(t) = \pi_Sd(i_Q(t)), \qquad t\in \Gamma(Q).
\end{eqnarray}
\begin{lemma}
The second fundamental forms $H$ and $K$ satisfy
\begin{eqnarray}
h_Q(Hs, t) = -h_S(s,Kt), \quad s\in S_x, t\in Q_x, \nonumber
\end{eqnarray}
for any $x\in Gr_p(\mathbb{C}^n)$.
\end{lemma}
\begin{proof}
See \cite{cite2}.
\end{proof}
\begin{lemma}
For a vector $w \in \mathbb{C}^n$, let $s=\pi_S(w)$ and $t=\pi_Q(w)$. Then
\begin{eqnarray}
\nabla^Ss = -K(t), \qquad \nabla^Qt = -H(s). \nonumber
\end{eqnarray}
\end{lemma}
\begin{proof}
Since $i_S(s)+i_Q(t)= ([g],w)$, we have
$$0=\pi _S\left( di_S(s)+di_Q(t) \right)= \nabla ^S(s) + K(t).$$
Thus $\nabla^Ss = -K(t)$. Similarly $\nabla^Qt = -H(s)$.
\end{proof}

Since $H$ is a $(1,0)$-form with values in  $S^*\otimes Q$, then $H$ can be regarded as a section of $T_{1,0}^*\otimes T_{1,0}$.
\begin{proposition}
$H$ can be regarded as the identity transformation of \ $T_{1,0}$
\end{proposition}

The unitary basis $w_1,\cdots,w_n$ of $\mathbb{C}^n$ provides us with the corresponding sections
\begin{eqnarray}
s_A = \pi_S(w_A)\in \Gamma (S), \quad t_A = \pi_Q(w_A)\in \Gamma (Q), \quad A = 1,\cdots ,n . \nonumber
\end{eqnarray}
\begin{proposition}
For arbitrary $(1,0)$-vectors $U$ and $V$, we have
\begin{eqnarray}
h_{Gr}(U,V) = \sum_{A=1}^n h_S(K_{\overline{V}}t_A, K_{\overline{U}}t_A) = \sum_{A=1}^n h_Q(H_Us_A, H_Vs_A). \nonumber
\end{eqnarray}
\end{proposition}

Proposition 2.3 and Proposition 2.4 were proved by the second author in \cite{cite4}.
\begin{remark}
From Lemma 2.1 and Proposition 2.4, we can express
\begin{eqnarray}
g_{Gr}(U,V) = -{\rm trace}_QH_UK_{\overline{V}} = -\overline{{\rm trace}_SK_{\overline{V}}H_U}.
\end{eqnarray}
\end{remark}

Since any vectors in $S_x$ (resp. $Q_x$) can be expressed by a linear combination of $s_1(x),\cdots ,s_n(x)$ (resp. $t_1(x),\cdots ,t_n(x)$), it follows from Lemma 2.2 that the curvature $R^S$ of $\nabla^{S}$ and $R^Q$ of $\nabla^Q$ are expressed by the following:
\begin{align}
 R^S(U,\overline{V})s_A &= \nabla^S_{U}(\nabla^Ss_A)(\overline{V}) - \nabla^S_{\overline{V}}(\nabla^Ss_A)(U) = K_{\overline{V}}H_Us_A, \\
 R^Q(U,\overline{V})t_A &= \nabla^Q_{U}(\nabla^Qt_A)(\overline{V}) - \nabla^Q_{\overline{V}}(\nabla^Qt_A)(U) = - H_UK_{\overline{V}}t_A.
\end{align}

It follows from $h_{Gr} = h_{s^*} \otimes h_Q$ that the curvature $R^{Gr}$ of the complex Grassmannian manifold can be expressed $R^{S^*} \otimes {\rm Id}_Q + {\rm Id}_{S^*} \otimes R^Q$. Thus we can compute $R^{Gr}$ as follows :
\begin{eqnarray}
R^{Gr}(U, \overline{V})Z = - H_ZK_{\overline{V}}H_U - H_UK_{\overline{V}}H_Z,
\end{eqnarray}
for $(1,0)$-vectors $U$, $V$, $Z$.

\begin{remark} Let us compute the holomorphic sectional curvature of $Gr_{n-1}(\mathbb{C}^n)$. Since the quotient bundle over $Gr_{n-1}(\mathbb{C}^n)$ is of rank 1, then it follows from the equations (2.3)  and (2.6) that
$$R^{Gr}(U,\overline{V})Z = - H_ZK_{\overline{V}}H_U - H_UK_{\overline{V}}H_Z. = h_{Gr}(Z,V)U + h_{Gr}(U,V)Z,$$
where $U,V$ is (1,0)-vectors. Thus for any unit (1,0)-vector $U$ we obtain
\begin{eqnarray}
{\rm Hol}^{Gr}(U) = h_{Gr}(R^{Gr}(U,\overline{U})U, U) = h_{Gr}(2U,U) = 2, \nonumber
\end{eqnarray}
where ${\rm Hol}^{Gr}(U)$ is the holomorphic sectional curvature of $Gr_{n-1}(\mathbb{C}^n)$.
\end{remark}


\section{Proof of Theorem 1.1}
We have a short exact sequence of holomorphic vector bundles:
\begin{eqnarray}
0 \rightarrow T_{1,0}M \rightarrow T_{1,0}Gr|_M \rightarrow N \rightarrow 0, \nonumber
\end{eqnarray}
where $T_{1,0}Gr|_M$ is a holomorphic vector bundle induced by $f$ from the holomorphic tangent bundle over $Gr_p(\mathbb{C}^n)$ and $N$ is a quotient bundle. In the same manner as in the previous section, we obtain second fundamental forms $\sigma$ and $A$ of $TM$ and $N$:
\begin{align}
\nabla ^{Gr}_UV &= \nabla ^M_UV + \sigma (U,V),& \quad &U\in T_{\mathbb{C}}M,\ V\in \Gamma (T_{1,0}M),\\
\nabla ^{Gr}_U\xi &= -A_{\xi}U + \nabla ^N_U\xi ,& &U\in T_{\mathbb{C}}M,\ \xi \in \Gamma (N).
\end{align}
For each point $x\in M$, $\sigma : T_{{1,0}_x}M\times T_{{1,0}_x}M\rightarrow N_x$ is a symmetric bilinear mapping. This is called the second fundamental form of $M$. The second fundamental form $A : N_x\times T_{{0,1}_x}M\rightarrow T_{{1,0}_x}M$ is a bilinear mapping. This is called the shape operator of $M$. We follow a convention of submanifold theory to define the shape operator. Second fundamental forms $\sigma$ and $A$ satisfy the following formulas.
\begin{formulas} 
\begin{itemize}
\item $\sigma (\overline{U},V)=0 ,\qquad A_{\xi}U=0$
\item $h_{Gr}\left( \sigma (U,V), \xi \right) = h_{Gr}\left( V, A_{\xi}\overline{U}\right)$
\item $h_{Gr}\left( R^M(U,\overline{V})Z,W \right) = h_{Gr}\left( R^{Gr}(U,\overline{V})Z,W\right) - h_{Gr}\left( \sigma (U,Z),\sigma (V,W)\right)$
\item $h_{Gr}\left( R^N(U,\overline{V})\xi ,\eta \right) = h_{Gr}\left( R^{Gr}(U,\overline{V})\xi ,\eta \right) + h_{Gr}\left( A_{\xi}\overline{V} ,A_{\eta}\overline{U}\right)$
\item $\left( \nabla _V\sigma \right) (U,Z) = \left( \nabla _U\sigma \right) (V,Z) $
\item $\left( \nabla _{\overline{V}}\sigma \right) (U,Z) = -\left( R^{Gr}(U,\overline{V})Z\right) ^{\perp} $
\end{itemize}
\end{formulas}
Note that the quotient bundle $N$ is isomorphic to the orthogonal bundle $T^{\perp}_{1,0}M$ as a $C^{\infty}$ complex vector bundle. The third, fourth and fifth formulas are called the equation of Gauss, the equation of Ricci and the equation of Codazzi respectively. From the equation of Codazzi,
$$ \nabla \sigma : T_{{1,0}_x}M \otimes T_{{1,0}_x}M \otimes T_{{1,0}_x}M \longrightarrow N_x$$
is a symmetric tensor for any $x\in M$.

We assume that $f^*Q\rightarrow M$ is projectively flat. $f^*Q\rightarrow M$ is projectively flat if and only if 
$$R^{f^*Q}(U,\overline{V}) = \alpha (U,\overline{V}) {\rm Id}_{Q_{f(x)}},\qquad {\rm for}\ U,V \in T_{1,0}M_x,$$
where $\alpha$ is a complex 2-form on $M$. Since $R^{f^*Q}$ is (1,1)-form, so is $\alpha$. It follows from the equation (2.3) that
$$h_{M}(U,V) = {\rm trace}R^Q(U,\overline{V}) = q\cdot \alpha (U,\overline{V}).$$
Therefore, $f^*Q\rightarrow M$ is projectively flat if and only if
\begin{equation}
R^{f^*Q}(U,\overline{V}) = \frac{1}{q} h_M(U,V){\rm Id}_{Q_{f(x)}},\qquad {\rm for}\ U,V \in T_{1,0}M_x. \tag{$*$} 
\end{equation}
\begin{remark}
It follows from the equation (2.5) that
\begin{equation}
R^{f^*Q}(U,\overline{V})= -H_UK_{\overline{V}}:Q_x \longrightarrow S_x \longrightarrow Q_x. \nonumber
\end{equation}
Therefore, if an immersion $f$ satisfies the equation $(\ast )$, the rank of $S$ is greater than or equal to that of $Q$.
\end{remark}

We denote by $\rm{Hol}$ the holomorphic sectional curvature of a K\"ahler manifold. By the equation of Gauss , if $U$ is a unit (1,0)-vector on $M$, then
\begin{align}
\begin{split}
{\rm Hol}^M(U) = h_M(R^M(U,\overline{U})U,U) &= h_{Gr}(R^{Gr}(U,\overline{U})U,U) - \| \sigma (U,U) \|^2  \\
&= {\rm Hol}^{Gr}(U) - \| \sigma (U,U) \|^2 .
\end{split}
\end{align}
\begin{lemma}
Under the assumption, for any unit $(1,0)$-vector $U$ on $M$ we have
\begin{eqnarray}
{\rm Hol}^{Gr} (U) = \frac{2}{q}. \nonumber
\end{eqnarray}
\end{lemma}
\begin{proof}
Let $U$ be a unit (1,0)-vector on $M$ at $x \in M$. By the equation $(\ast )$, we have
\begin{eqnarray}
-H_UK_{\overline{U}} = R^{f^*Q}(U, \overline{U}) = \frac{1}{q} {\rm Id}_{Q_x}.
\end{eqnarray}
It follows from equations (2.6) and (3.4) that
\begin{align}
{\rm Hol}^{Gr}(U) &= h_{Gr}(R^{Gr}(U, \overline{U})U, U) =-2h_{S^* \otimes Q}(H_UK_{\overline{U}}H_U, H_U) \nonumber \\
&= \frac{2}{q}h_{S^* \otimes Q}(H_U, H_U) = \frac{2}{q}. \nonumber
\end{align}
\end{proof}
\begin{lemma}
Under the assumption, for any $(0,1)$-vector $\overline{V}$ on $M$ we have
\begin{eqnarray}
\nabla^M_{\overline{V}}\sigma = 0. \nonumber
\end{eqnarray}
\end{lemma}
\begin{proof}
By the equation (2.6) and the equation $(\ast )$, it follows that
\begin{align}
R^{Gr}(U,\overline{V})Z &= - H_ZK_{\overline{V}}H_U - H_UK_{\overline{V}}H_Z \nonumber \\
&= \frac{1}{q} h_{Gr}(Z,V)U + \frac{1}{q} h_{Gr}(U,V)Z. \nonumber
\end{align}
By the equation of Codazzi, we have
$$\nabla^M_{\overline{V}}\sigma(U,Z) = - (R^{Gr}(U,\overline{V})Z)^{\perp} = 0.$$
\end{proof}

We use the complexification of Ros' lemma in \cite{cite6}.
\begin{lemma}
Let $T$ be a $(p,q)$-covariant tensor on an $m$-dimensional compact K\"{a}hler manifold $M$. Then
\begin{eqnarray}
\int_{UM}(\nabla T)(\overline{U}, U, \cdots , U, \overline{U},\cdots , \overline{U})dU = 0, \nonumber
\end{eqnarray}
where $UM=\{ U\in T_{1,0}M\ |\ \| U\| =1\}$.
\end{lemma}

\begin{proof}
We regard $M$ as a $2m$-dimensional compact Riemannian manifold with the almost complex structure $J$. If $U$ is a (1,0)-vector on $M$, then there is a real vector $X$ on $M$ such that
\begin{eqnarray}
U = \frac{1}{\sqrt{2}} (X-\sqrt{-1}JX),\quad \overline{U} = \frac{1}{\sqrt{2}} (X+\sqrt{-1}JX). \nonumber
\end{eqnarray}
Then we can define real $k$-covariant tensors $K$, $L$ by the following:
\begin{align}
2K(X_1,\cdots ,X_k) &= T(U_1,\cdots , U_p, \overline{U_{p+1}},\cdots , \overline{U_k}) + \overline{ T(U_1,\cdots , U_p, \overline{U_{p+1}},\cdots , \overline{U_k})} , \nonumber \\
2L(X_1,\cdots ,X_k) &= \sqrt{-1}\{ T(U_1,\cdots , U_p, \overline{U_{p+1}},\cdots , \overline{U_k}) + \overline{ T(U_1,\cdots , U_p, \overline{U_{p+1}},\cdots , \overline{U_k})}\} , \nonumber 
\end{align}
where $k=p+q$, $U_1,\cdots ,U_k$ are $(1,0)$-vectors corresponding to $X_1,\cdots X_k$. Then $T$, $K$ and $L$ satisfy the following equation:
\begin{equation}
T(U_1,\cdots , U_p, \overline{U_{p+1}},\cdots , \overline{U_k}) = K(X_1,\cdots ,X_k) - \sqrt{-1}L(X_1,\cdots ,X_k).
\end{equation}
We get the covariant derivative of both sides of the equation (3.5):
\begin{equation}
\begin{split}
(\nabla _{\overline{U}}T)(U,\cdots ,\overline{U},\cdots ) &= \frac{1}{\sqrt{2}}(\nabla _{X+\sqrt{-1}JX}K)(X,\cdots ,X) \\
&- \frac{\sqrt{-1}}{\sqrt{2}}(\nabla _{X+\sqrt{-1}JX}L)(X,\cdots ,X).
\end{split}
\end{equation}
Since the covariant derivative is linear, then
\begin{equation}
(\nabla _{X+\sqrt{-1}JX}K)(X,\cdots ,X) = (\nabla _XK)(X,\cdots ,X) +\sqrt{-1}(\nabla _{JX}K)(X,\cdots ,X). 
\end{equation}

We use the original lemma of A.Ros \cite{cite6}.
\begin{lemma}[A.Ros]
Let $T$ be a $k$-covariant tensor on a compact Riemannian manifold $M$. Then
\begin{eqnarray}
\int_{UM}(\nabla T)(X, \cdots , X)dX = 0 \nonumber
\end{eqnarray}
\end{lemma}
For a proof, see \cite{cite6}. From equations (3.6), (3.7) and Lemma 3.5, we get the result.
\end{proof}

\begin{proof}[Proof of the Theorem.]
We define a (2,2)-covariant tensor $T$ on $M$ by
\begin{eqnarray}
T(U,V,\overline{Z},\overline{W}) = h_{Gr}\left( \sigma (U,V) , \sigma (Z,W) \right) ,
\end{eqnarray}
where $U$, $V$, $Z$, $W$ are (1,0)-vectors  on $M$. Using the equation of Ricci and the equation of Codazzi, we obtain
\begin{eqnarray}
(\nabla^2T)(\overline{U},U,U,U,\overline{U},\overline{U}) = h_M\left( (\nabla^2 \sigma )(\overline{U}, U, U, U), \sigma (U, U) \right) + \| (\nabla \sigma)(U, U, U) \|^2. \nonumber
\end{eqnarray}
Using the  Ricci identity, we obtain
$$(\nabla^2 \sigma)(U, \overline{U}, U, U) - (\nabla^2 \sigma)(\overline{U}, U, U, U) = R^{\perp}(U, \overline{U}) \left( \sigma (U, U) \right) - 2 \sigma \left( R^M(U, \overline{U})U, U \right).$$
It follows from Lemma 3.3 that
\begin{eqnarray}
(\nabla ^2 \sigma )(\overline{U}, U, U, U) = - R^{\perp}(U, \overline{U}) \left( \sigma (U, U) \right) + 2\sigma \left( R^M (U, \overline{U})U, U \right). \nonumber
\end{eqnarray}
Therefore, we obtain
\begin{align}
\\
(\nabla^2T)(\overline{U},U,U,U,\overline{U},\overline{U}) &= -h_{Gr}\left( R^{\perp}(U, \overline{U})(\sigma (U, U) ), \sigma (U, U) \right) \nonumber \\
&+ 2 h_{Gr} \left( \sigma ( R^M(U, \overline{U})U, U), \sigma (U, U) \right) + \| (\nabla \sigma ) (U, U, U) \| ^2. \nonumber
\end{align}
From the equatin of Ricci, we have
\begin{eqnarray}
h_{Gr}\left( R^{\perp}(U,\overline{U})(\sigma (U,U)), \sigma (U,U) \right) = h_{Gr}\left( R^{Gr}(U,\overline{U}) (\sigma (U, U)),\sigma (U,U) \right) + \| A_{\sigma(U,U)} \overline{U} \|^2. \nonumber
\end{eqnarray}
Using the second fundamental forms $H$ and $K$ of vector bundles, we can compute their equations more:
\begin{align}
\\
h_{Gr}\left( R^{\perp}(U,\overline{U})(\sigma (U,U)), \sigma (U,U) \right) &= h_{Gr}(-H_{\sigma (U,U)}K_{\overline{U}}H_U,H_{\sigma (U,U)}) \nonumber \\ 
&+ h_{Gr}(-H_UK_{\overline{U}}H_{\sigma (U,U)},H_{\sigma (U,U)}) + \| A_{\sigma (U,U)}\overline{U} \| ^2. \nonumber
\end{align}
In the following calculation, we extend (1,0)-vectors to local holomorphic vector fields if it is necessary.

\begin{lemma}
The covariant derivative of $R^{f^*Q}$ by a (1,0)-vector is identically zero if and only if  $H_{\sigma (U,U)}K_{\overline{U}}=0$ for any (1,0)-vector $U$.
\end{lemma}
\begin{proof}
We have
\begin{align}
\left( \nabla_Z R^{f^*Q}\right) (U,\overline{V}) = -\nabla_Z (H_UK_{\overline{V}}) + H_{\nabla_Z U}K_{\overline{V}} = -(\nabla_Z H)(U)K_{\overline{V}}. \nonumber
\end{align}
Since we can easily show that $H_{\sigma (U,U)} = (\nabla_U H)(U)$, we obtain
\begin{equation*}
-H_{\sigma (U,U)}K_{\overline{U}} = (\nabla_U H)(U)K_{\overline{U}} = \left( \nabla_U R^{f^*Q}\right) (U,\overline{U}) =0. 
\end{equation*}
\end{proof}

It follows from the assumption that
\begin{align}
\left( \nabla_Z R^{f^*Q}\right) (U,\overline{V}) &= \nabla_Z \left( R^{f^*Q}(U,\overline{V})\right) - R^{f^*Q}(\nabla_Z U,\overline{V}) \nonumber \\
&= \frac{1}{q}\nabla_Z \left( h_M(U,V) \right) {\rm Id}_Q - \frac{1}{q}h_{Gr}(\nabla_Z U,V){\rm Id}_Q \nonumber \\
&= \frac{1}{q}h_{Gr}(\nabla_Z U,V){\rm Id}_Q - \frac{1}{q}h_{Gr}(\nabla_Z U,V){\rm Id}_Q = 0, \nonumber
\end{align}
where $U$, $V$, $Z$ are (1,0)-vectors on $M$. Then it follows from the Lemma 3.6 and the equation (3.10) that
\begin{align}
\\
h_{Gr}\left( R^{\perp}(U,\overline{U})(\sigma (U,U)), \sigma (U,U) \right) &= h_{Gr}(-H_UK_{\overline{U}}H_{\sigma (U,U)},H_{\sigma (U,U)}) +\| A_{\sigma (U,U)}\overline{U} \| ^2 \nonumber \\
&= \frac{1}{q}\| \sigma (U,U) \| ^2 + \| A_{\sigma (U,U)} \overline{U} \| ^2. \nonumber
\end{align}
Using the equation of Gauss, we have
\begin{align}
\\
h_{Gr}\left( \sigma (R^M(U,\overline{U})U,U),\sigma (U,U) \right) &= h_{Gr}\left( R^M(U,\overline{U})U,A_{\sigma (U,U)}\overline{U} \right) \nonumber \\
&= h_{Gr}\left( R^{Gr}(U,\overline{U})U,A_{\sigma (U,U)}\overline{U}\right) - \| A_{\sigma (U,U)}\overline{U} \| ^2 \nonumber \\
&= -2 h_{Gr}(H_UK_{\overline{U}}H_U,H_{A_{\sigma (U,U)}\overline{U}}) - \| A_{\sigma (U,U)}\overline{U} \| ^2 \nonumber \\
&= \frac{2}{q} \| \sigma (U,U) \| ^2 - \| A_{\sigma (U,U)}\overline{U} \| ^2. \nonumber
\end{align}
Combining the equations (3.11) and (3.12) with (3.9), we obtain
\begin{align}
\\
(\nabla^2 T)(\overline{U},U,U,U,\overline{U},\overline{U}) &= -\left( \frac{1}{q} \| \sigma (U,U) \| ^2 + \| A_{\sigma (U,U)} \overline{U} \| ^2 \right)  \nonumber \\
&+ 2 \left(  \frac{2}{q} \| \sigma (U,U) \| ^2 - \| A_{\sigma (U,U)}\overline{U} \| ^2 \right) + \| (\nabla \sigma) (U,U,U) \| ^2 \nonumber \\
&= \frac{3}{q} \left( \| \sigma (U,U) \| ^2 - q \| A_{\sigma (U,U)}\overline{U} \| ^2 \right) + \| (\nabla \sigma) (U,U,U)\| ^2. \nonumber
\end{align}

By integrating both sides of the equation (3.13), Lemma 3.4 yields
\begin{eqnarray}
\\
\frac{3}{q}\int_{UM} \left( \| \sigma (U,U) \| ^2 - q \| A_{\sigma (U,U)}\overline{U} \| ^2 \right) dU +\int_{UM}\| (\nabla \sigma) (U,U,U)\| ^2 dU =0. \nonumber
\end{eqnarray}

From now on we assume that the holomorphic sectional curvature of $M$ is greater than or equal to $\frac{1}{q}$. 

Let us compute the first term of the left hand side of the equation (3.14). We define $\xi \in N$ as $\sigma (U,U) = \| \sigma (U,U)\| \xi .$ Then we have
$$A_{\sigma (U,U)}\overline{U} = \| \sigma (U,U) \| A_{\xi}\overline{U}.$$

We denote by $\tau$ the real structure on the complexification $T_{\mathbb{C}}M$ of the tangent bundle over $M$. Let $B:=A_{\xi}\circ \tau$. $B$ is an anti-linear transformation and satisfies the following equation:
$$h_{Gr}(BU,V) = h_{Gr}(BV,U),\qquad {\rm for}\ U,V \in T_{{1,0}_x}M,\ x\in M.$$
If we regard $B$ as a real linear transformation on the real vector space  with an inner product $\mathfrak{Re}(h_{Gr})$ , then $B$ is a symmetric transformation. Thus $B$ has the nonnegative maximum eigenvalue $\lambda$ and we denote by $e$ the corresponding unit eigenvector. By Cauchy-Schwarz inequality, we have
$$\lambda = h_{Gr}(Be,e) = h_{Gr}(A_{\xi}\overline{e},e) = h_{Gr}(\xi , \sigma (e,e)) \leq \| \sigma (e,e) \| .$$
It follows from the equation (3.3), the lemma 3.2 and the hypothesis that
\begin{eqnarray}
\| A_{\xi}\overline{U} \| ^2 \leq \lambda ^2 \leq \| \sigma (e,e) \| ^2 \leq \frac{1}{q}. \nonumber
\end{eqnarray}
It follows that
\begin{align}
\| \sigma (U,U) \| ^2 - q \| A_{\sigma (U,U)}\overline{U}\| ^2 &= \| \sigma (U,U) \| ^2(1-q\| A_{\xi}\overline{U} \| ^2) \nonumber \\
&\geq \| \sigma (U,U) \| ^2(1-q\cdot \frac{1}{q}) \geq 0. \nonumber
\end{align}

Thus it follows from the equation (3.14) that
$$\| (\nabla \sigma )(U,U,U)\| ^2 = 0.$$
Since $\nabla \sigma $ is a symmetric tensor, $\nabla \sigma$ vanishes.

Conversely, we assume that $M$ has the parallel second fundamental form. From the equation (3.3) and the Lemma 3.3, it is enough to prove that $\| \sigma (U,U) \| ^2\leq \frac{1}{q}$, where $U$ is an arbitrary unit $(1,0)$-vector on $M$. Let $T$ be a $(2,2)$-covariant tensor on $M$ defined by the equation (3.8). Since the second fundamental form $\sigma$ is parallel, $T$ is also parallel and so $\nabla^2T = 0$. It follows from the equation (3.13) that
\begin{equation}
\| \sigma (U,U) \| ^2 - q \| A_{\sigma (U,U)}\overline{U} \| ^2 = 0.
\end{equation}
Cauchy-Schwarz inequality and the equation (3.15) imply that
\begin{align}
\| \sigma (U,U) \| ^2 &= h_{Gr}\left( \sigma (U,U), \sigma(U,U) \right) = h_{Gr}\left( U, A_{\sigma(U,U)}\overline{U} \right) \nonumber \\
&\leq \| A_{\sigma(U,U)}\overline{U} \| = \frac{1}{\sqrt{q}}\| \sigma (U,U) \| . \nonumber
\end{align}
Therefore, $\| \sigma (U,U) \| ^2 \leq \frac{1}{q}$.

\end{proof}


\bibliographystyle{amsplain}

\end{document}